\documentclass[runningheads,a4paper]{llncs}
\usepackage{amssymb,amsmath,amsfonts}
\usepackage{graphicx}
\usepackage{url}
\usepackage{enumerate}
\usepackage{lscape}
\usepackage{multirow}

\urldef{\mailsa}\path|devsi.bantva@gmail.com|
\newcommand{\keywords}[1]{\par\addvspace\baselineskip
\noindent\keywordname\enspace\ignorespaces#1}

\newcommand{\be}{\begin{equation}}
\newcommand{\ee}{\end{equation}}
\newcommand{\bea}{\begin{eqnarray}}
\newcommand{\eea}{\end{eqnarray}}
\newcommand{\bean}{\begin{eqnarray*}}
\newcommand{\eean}{\end{eqnarray*}}

\def\diam{{\rm diam}}
\def\Span{{\rm span}}
\def\rn{{\rm rn}}

\def\vp{\varphi}

\def\dis{\displaystyle}

\begin{document}
\mainmatter  

\title{A lower bound for the radio number of graphs}

\author{Devsi Bantva\inst{1}}
\institute{Lukhdhirji Engineering College, Morvi - 363 642 \\ Gujarat (INDIA) \\
\mailsa\\
}

\toctitle{A lower bound for the radio number of graphs}
\tocauthor{Devsi Bantva}
\maketitle

\begin{abstract}
A radio labeling of a graph $G$ is a mapping $\vp : V(G) \rightarrow \{0, 1, 2,...\}$ such that $|\vp(u)-\vp(v)|\geq \diam(G) + 1 - d(u,v)$ for every pair of distinct vertices $u,v$ of $G$, where $\diam(G)$ and $d(u,v)$ are the diameter of $G$ and distance between $u$ and $v$ in $G$, respectively. The radio number $\rn(G)$ of $G$ is the smallest number $k$ such that $G$ has radio labeling with $\max\{\vp(v):v \in V(G)\}$ = $k$.  In this paper, we slightly improve the lower bound for the radio number of graphs given by Das \emph{et al.} in \cite{Das} and, give necessary and sufficient condition to achieve the lower bound. Using this result, we determine the radio number for cartesian product of paths $P_{n}$ and the Peterson graph $P$. We give a short proof for the radio number of cartesian product of paths $P_{n}$ and complete graphs $K_{m}$ given by Kim \emph{et al.} in \cite{Kim}.
\keywords{Radio labeling, radio number, cartesian product of graphs, Peterson graph.}
\end{abstract}

\section{Introduction}
A radio labeling is a distance constrained graph labeling problem originated from well known channel assignment problem. In channel assignment problem, a set of radio stations is given and the task is to assign channels to each radio station such that interference is minimum with optimum use of spectrum. It is known that the interference constraint relies on the distance between two radio stations. The interference between radio stations increases as distance between them decreases and vice-versa. This problem is modeled by graphs: Radio stations are represented by vertices of graphs and interference level is related with distance between them. The assignment of channels is converted into graph labeling problem. Motivated through this, Chartrand \emph{et al.} introduced the concept of radio labeling in \cite{Chartrand1,Chartrand2} as follows:
\begin{definition}
{\em
A \emph{radio labeling} of a graph $G$ is a mapping $\vp: V(G) \rightarrow \{0, 1, 2, \ldots\}$ such that for every pair of distinct vertices $u, v$ of $G$,
\be\label{eqn:def}
d(u,v) + |\vp(u)-\vp(v)| \geq \diam(G) + 1.
\ee
The integer $\vp(u)$ is called the \emph{label} of $u$ under $\vp$, and the \emph{span} of $\vp$ is defined as $\Span(\vp) = \max \{|\vp(u)-\vp(v)|: u, v \in V(G)\}$. The \emph{radio number} of $G$ is defined as
$$
\rn(G) := \min_{\vp} \{\Span(\vp)\}
$$
with minimum taken over all radio labelings $\vp$ of $G$. A radio labeling $\vp$ of $G$ is \emph{optimal} if $\Span(\vp) = \rn(G)$.
}
\end{definition}

Note that an optimal radio labeling always assign 0 to some vertex and in this case, the span of $\vp$ is the maximum integer assign by $\vp$. A radio labeling is a one-to-one integral function on $V(G)$ to the set of non-negative integers and hence it induces an ordering $x_{0},x_{1},...,x_{p-1}$ $(p = |V(G)|)$ of $V(G)$ such that $0 = \vp(x_{0}) < \vp(x_{1}) < ... < \vp(x_{p-1}) = \Span(\vp)$. It is clear that if $\vp$ is an optimal radio labeling of graph $G$ and $\psi$ is any other radio labeling of $G$ then $\Span(\vp) \leq \Span(\psi)$.

A radio labeling problem is recognized as one of the tough graph labeling problems. In most of the research papers, the trend is to determine the radio number for specific graph families. A very few research papers are on general cases which gives lower bound for the radio number of trees and arbitrary graphs. These are as follows: In \cite{Daphne1}, Liu gave a lower bound for the radio number of trees and presented a class of trees, namely spiders, achieving this lower bound. In \cite{Bantva1,Bantva2}, Bantva \emph{et al.} presented this lower bound using different notations and gave a necessary and sufficient condition to achieve this lower bound. Using this result they determined the radio number for banana trees, firecrackers trees and a special class of caterpillars. Recently, in \cite{Das}, Das \emph{et al.} gave a technique to find a lower bound for the radio number of any graphs.

In this paper, our focus is on a lower bound for the radio number of graphs. We slightly improve the technique to find a lower for the radio number of graphs given by Das \emph{et al.} in \cite{Das} and, give a necessary and sufficient condition to achieve the lower bound. Our results are also useful to determine the radio number of graphs when it is slightly more than the lower bound for the radio number of graphs (see case of cartesian product of paths $P_{n}$ with the Peterson graph $P$ and complete graphs $K_{m}$ when $n$ is odd). We determine the radio number for cartesian product of paths $P_{n}$ and the Peterson graph $P$ and, give a short proof for the radio number of cartesian product of paths $P_{n}$ and complete graphs $K_{m}$ given by Kim \emph{et al.} in \cite{Kim}.

\section{A lower bound for the radio number of graphs}
In this section, we slightly improve the technique to find a lower bound for the radio number of graphs given by Das \emph{et al.} in \cite{Das} and make it more effective (more detail is given in concluding remarks). We also give a necessary and sufficient condition to achieve the lower bound.

Let $G$ = ($V, E$) be a simple connected graph without loops and multiple edges. We denote the vertex set of $G$ by $V(G)$. We assume $|V(G)|$ = $p$ throughout this paper. The distance between two vertices $u$ and $v$, denoted by $d(u,v)$, is the least length of a path joining $u$ and $v$. The diameter of a graph $G$, denoted by $\diam(G)$ (or simply $d$ to use in equations), is $\max\{d(u,v) : u, v \in V(G)\}$. Let $S$ be a induced subgraph of $G$, then for any $v \in V(G)$, $d(v,S)$ = $\min\{d(v,w) : w \in S\}$ and $\diam(S)$ = $\max\{d(u,v) : u,v \in S\}$. Denote [0, $n$] = \{$0,1,2,...,n$\}. We follow \cite{West} for standard graph theoretic definition and notation which are not defined here.

Let $H$ be an induced subgraph of connected graph $G$. The choice of induced subgraph $H$ in $G$ is very crucial in our discussion. In fact, the choice of $H$ plays an important role and key idea of our philosophy to improve a lower bound for the radio number of graphs. But at this moment, we provide only the following information about $H$ and postpone the detail discussion about it till the end. We choose a subgraph $H$ of $G$ such that $\diam(H)$ = $k$. We set $L_{0}$ = $V(H)$. Let $N(L_{0})$ denote the set of vertices which are adjacent to vertices of $L_{0}$. Set $L_{1}$ = $N(L_{0}) \setminus L_{0}$. Recursively define $L_{i+1}$ = $N(L_{i}) \setminus (L_{0} \cup ... \cup L_{i})$. Assume that the maximum value of index $i$ for $L_{i}$ is $h$ known as maximum level. Since $G$ is connected it is clear that $L_{s} \neq \phi$ for $0 \leq s \leq h$ and $L_{t} = \phi$ for $t > h$. We fix these sets for rest of discussion.

Let $\vp$ be any radio labeling of $G$ with $\Span(\vp)$ = $n$. Note that the function $\vp$ is injective but not surjective. Since $\vp$ is injective, it induces an ordering $x_{0},x_{1},....,x_{p-1}$ of $V(G)$ with 0 = $\vp(x_{0}) < \vp(x_{1}) <...< \vp(x_{p-1})$. Assume that the assign labels are {$a_{0},a_{1},...,a_{p-1}$} such that $\vp(x_{i})$ = $a_{i}$, $0 \leq i \leq p-1$ then 0 = $a_{0} < a_{1} < ... < a_{p-1} = \Span(\vp) = n$. Since $\vp$ is not surjective, it is clear that $\vp(V(G)) = \{a_{0},a_{1},...,a_{p-1}\} \subset [0,n]$. The labels assigned by $\vp$ to vertices of $G$ are called \emph{used labels} and the labels $[0,n] \setminus \{a_{0},a_{1},...,a_{p-1}\}$ are called \emph{unused labels}. So to give a lower bound, our aim is to count both the used and unused labels.

The number of unused labels are the sum of $a_{t+1}-a_{t}-1$, where $t$ varies from 0 to $p-2$. Using definition of radio labeling and triangle inequality twice for $d(x_{t+1},x_{t})$ in $G$, we obtain
\bea
a_{t+1}-a_{t}-1 & \geq & d+1-d(x_{t+1},x_{t})-1 \label{eqn:tri1} \\
& \geq & d+1-d(x_{t+1},L_{0})-d(x_{t},L_{0})-\diam(L_{0})-1 \label{eqn:tri2} \\
& = & d+1-d(x_{t+1},L_{0})-d(x_{t},L_{0})-k-1. \nonumber
\eea
Summing this latter inequality for 0 to $p-2$, we obtain the total number of unused labels. Thus the number of unused labels is at least
\bean
\dis\sum_{t=0}^{p-2}(a_{t+1}-a_{t}-1) & \geq & \dis\sum_{t=0}^{p-2}(d+1-d(x_{t+1},L_{0})-d(x_{t},L_{0})-k-1) \\
& = & (p-1)(d-k)-2\dis\sum_{t=0}^{p-2}d(x_{t},L_{0}) + d(x_{0},L_{0}) + d(x_{p-1},L_{0}) \\
& = & (p-1)(d-k)+d(x_{0},L_{0}) + d(x_{p-1},L_{0})-2\dis\sum_{i=0}^{h}|L_{i}|i.
\eean
Note that as the label set for radio labeling includes 0 as well, the used labels have an additive factor of $-1$. Hence, the sum of used and unused labels is at least as follows.
\bean
\Span(\vp) & \geq & p - 1 + (p-1)(d-k)+d(x_{0},L_{0}) + d(x_{p-1},L_{0})-2\dis\sum_{i=0}^{h}|L_{i}|i \\
& \geq & (p-1)(d-k+1)+d(x_{0},L_{0}) + d(x_{p-1},L_{0})-2\dis\sum_{i=0}^{h}|L_{i}|i.
\eean
Note that $d(x_{0},L_{0})+d(x_{p-1},L_{0})$ has minimum value if $x_{0}, x_{p-1} \in L_{0}$ when $|L_{0}| \geq 2$ and $x_{0} \in L_{0}, x_{p-1} \in L_{1}$ when $|L_{0}| = 1$. Define $\delta$ = 0 if $|L_{0}| \geq 2$ and 1 if $|L_{0}| = 1$. Hence, we obtain
\bean \rn(G) \geq (p-1)(d-k+1)+\delta-2\dis\sum_{i=0}^{h}|L_{i}|i. \eean

We now come to the selection of $H$ as an induced subgraph of $G$. We choose an induced subgraph $H$ in $G$ such that the set of vertices $V(G) \setminus V(H)$ can be partitioned into distinct sets $V_{1},V_{2},...,V_{m} (m \geq 2)$ and when we fix $V(H)$ as $L_{0}$ then it possible to order the vertices of $G$ as $x_{0},x_{1},...,x_{p-1}$ such that $d(x_{i},x_{i+1})$ satisfies the equation $d(x_{i},x_{i+1}) = d(x_{i},L_{0})+d(x_{i+1},L_{0})+\diam(L_{0})$, where $x_{i} \in V_{i}, x_{i+1} \in V_{j}, i \neq j$ or, one or both of $x_{i}, x_{i+1}$ is in $V(H)$. We also keep in mind that such an ordering $x_{0},x_{1},...,x_{p-1}$ satisfies conditions $d(x_{0},L_{0}) = 0$, $d(x_{p-1},L_{0})$ = 1 when $|L_{0}| = 1$ and $d(x_{0},L_{0}) = d(x_{p-1},L_{0}) = 0$ when $|L_{0}| \geq 2$. We also inform the readers that in case of trees, the set of weight center(s) $W(T)$ (see \cite{Daphne1} and \cite{Bantva2} for definition and detail about weight center) is always a good choice as $L_{0}$ and more useful results are given in \cite{Daphne1} and \cite{Bantva1,Bantva2} to determine the radio number of trees than the technique discussed above. We advised the readers to refer \cite{Daphne1} and \cite{Bantva1,Bantva2} for the radio number of trees.

Finally, from above discussion, we summarize our result as follows.

\begin{theorem}\label{thm:lower} Let $G$ be a simple connected graph of order $p$, diameter $d$ and $L_{i}$'s, $\delta$ are defined as earlier. Denote $\diam(L_{0})$ = $k$. Then
\be\label{eqn:lower} \rn(G) \geq (p-1)(d-k+1)+\delta-2\dis\sum_{i=0}^{h}|L_{i}|i. \ee
\end{theorem}

\begin{theorem}\label{thm:main} Let $G$ be a simple connected graph of order $p$, diameter $d$ and $L_{i}$'s, $\delta$ are defined as earlier. Denote $\diam(L_{0})$ = $k$. Then
\be\label{eqn:main} \rn(G) = (p-1)(d-k+1)+\delta-2\dis\sum_{i=0}^{h}|L_{i}|i \ee
holds if and only if there exist a radio labeling $\vp$ with $0 = \vp(x_{0}) < \vp(x_{1}) <...< \vp(x_{p-1}) = \Span(\vp) = \rn(G)$ such that all the following hold for $0 \leq i \leq p-1$:
\begin{enumerate}[\rm (a)]
  \item $d(x_{i},x_{i+1})$ = $d(x_{i},L_{0})+d(x_{i+1},L_{0})+k$,
  \item $x_{0}, x_{p-1} \in L_{0}$ if $|L_{0}| \geq 2$ and $x_{0} \in L_{0}, x_{p-1} \in L_{1}$ if $|L_{0}| = 1$,
  \item $\vp(x_{0})$ = 0 and $\vp(x_{i+1})$ = $\vp(x_{i})+d+1-d(x_{i},L_{0})-d(x_{i+1},L_{0})-k$.
\end{enumerate}
\end{theorem}
\begin{proof}\textsf{Necessity:} Suppose that (\ref{eqn:main}) holds. Then there exist an optimal radio labeling $\vp$ of $G$ with $\Span(\vp)$ = $(p-1)(d-k+1)+\delta-2\sum_{i=0}^{h}|L_{i}|i$. Let $x_{0},x_{1},...,x_{p-1}$ with 0 = $\vp(x_{0}) < \vp(x_{1}) < ... < \vp(x_{p-1}) = \Span(\vp)$ is an ordering of $V(G)$ induced by $\vp$. Note that $\Span(\vp)$ = $(p-1)(d-k+1)+\delta-2\sum_{i=0}^{h}|L_{i}|i$ is possible if equalities hold in (\ref{eqn:tri1}) and (\ref{eqn:tri2}) together with $x_{0},x_{p-1} \in L_{0}$ when $|L_{0}| \geq 2$ and, $x_{0} \in L_{0}, x_{p-1} \in L_{1}$ when $|L_{0}| = 1$. Note that equalities in (\ref{eqn:tri1}) and (\ref{eqn:tri2}) gives $d(x_{i},x_{i+1})$ = $d(x_{i},L_{0})+d(x_{i+1},L_{0})+k$. These all together turn the definition of radio labeling (\ref{eqn:def}) as $\vp(x_{0})$ = 0 and $\vp(x_{i+1})$ = $\vp(x_{i})+d+1-L(x_{i})-L(x_{i+1})-k$.

\textsf{Sufficiency:} Suppose that there exist a radio labeling $\vp$ with $0 = \vp(x_{0}) < \vp(x_{1}) <...< \vp(x_{p-1}) = \Span(\vp) =  \rn(G)$ such that (a), (b) and (c) holds. It is enough to prove that $\Span(\vp)$ = $(p-1)(d-k+1)+\delta-2\sum_{i=0}^{h}|L_{i}|i$. From (b) and (c), we have
\bean
\Span(\vp) & = & \vp(x_{p-1}) - \vp(x_{0}) \\
& = & \dis\sum_{t=0}^{p-2}\bigg(\vp(x_{t+1})-\vp(x_{t})\bigg) \\
& = & \dis\sum_{t=0}^{p-2}\bigg(d+1-d(x_{t+1},L_{0})-d(x_{t},L_{0})-k\bigg) \\
& = & (p-1)(d-k+1)-2\dis\sum_{t=0}^{p-2}d(x_{t},L_{0})+L(x_{0},L_{0})+d(x_{p-1},L_{0}) \\
& = & (p-1)(d-k+1)+\delta-2\dis\sum_{i=0}^{h}|L_{i}|i
\eean
which completes the proof.
\end{proof}

\begin{remark}{\rm As a consequence of above Theorem \ref{thm:main}, we obtain that if one or more conditions of Theorem \ref{thm:main} does not hold then
\be \rn(G) > (p-1)(d-k+1)+\delta-2\dis\sum_{i=0}^{h}|L_{i}|i \ee}
\end{remark}

\section{Radio number for some cartesian product of two graphs}
In this section, we continue to use the terminology and notation defined in previous section. We determine the radio number for cartesian product of paths $P_{n}$ and the Peterson graph $P$ using results of previous section. We present a short proof for the radio number of cartesian product of paths $P_{n}$ and complete graphs $K_{m}$ given by Kim \emph{et al.} in \cite{Kim} using our results approach.

Let $G$ = ($V(G), E(G)$) and $H$ = ($V(H), E(H)$) be two graphs. The cartesian product of $G$ and $H$, denoted by $G \square H$, is the graph $G_{\square}$ = ($V(G_{\square}), E(G_{\square})$) where $V(G_{\square})$ = $V(G) \times V(H)$ and two vertices ($a,b$) and ($c,d$) are adjacent if $a$ = $c$ and ($b,d$) $\in E(H)$ or $b$ = $d$ and ($a,c$) $\in E(G)$.

\subsection{Radio number for $P_{n} \square P$}
The peterson graph, denoted by $P$, is the complement of the line graph of complete graph $K_{5}$. The peterson graph and, the cartesian product of a path $P_{5}$ and the Peterson graph $P$ is shown in Fig. 1 and 2, respectively. Note that $|P_{n} \square P|$ = $|P_{n}| \times |P|$ = $10n$ and $\diam(P_{n} \square P)$ = $n+1$. We denote the vertex set of $P_{n}$ by $V(P_{n})$ = $\{u_{1},u_{2},...,u_{n}\}$ with $(u_{i},u_{i+1}) \in E(P_{n}), 1 \leq i \leq n-1$ and the vertex set of $P$ by $V(P)$ = $\{v_{1},v_{2},...,v_{10}\}$ with $E(P)$ = \{$v_{i}v_{i+1}$, $v_{1}v_{6}$, $v_{1}v_{8}$, $v_{2}v_{7}$, $v_{3}v_{9}$, $v_{4}v_{8}$, $v_{5}v_{7}$, $v_{6}v_{9}$, $v_{7}v_{10}$, $v_{8}v_{10}$, $v_{9}v_{10}$ : $1 \leq i \leq 5$\}.
\begin{figure}[h!]
  \centering
  \includegraphics[width=4.5in]{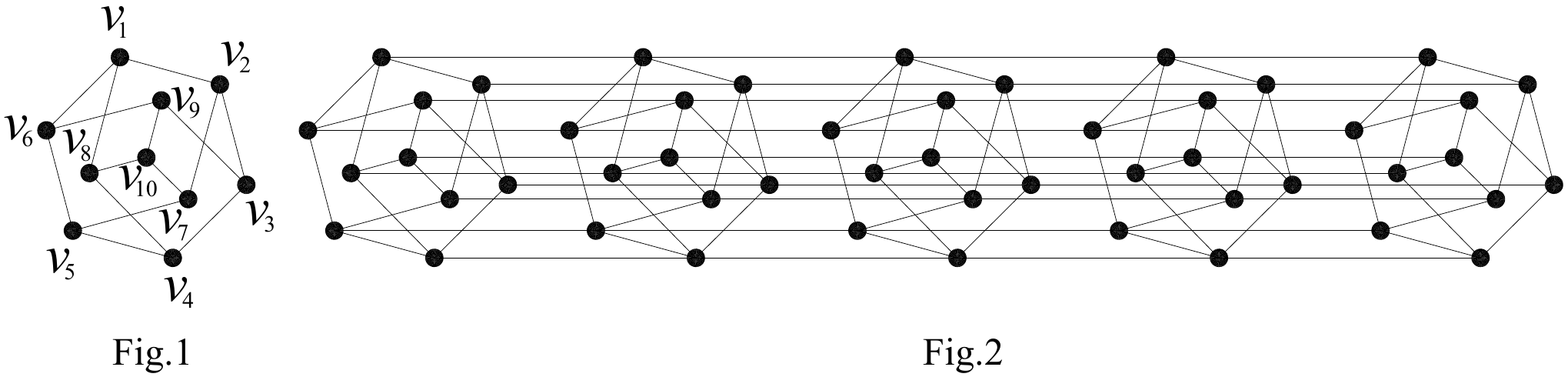}
\end{figure}
\begin{theorem}\label{thm:pet} Let $n \geq 3$ be an integer. Then
\begin{equation}\label{eqn:pnkm}
\rn(P_{n} \square P) := \left\{
\begin{array}{ll}
5n^{2}-n+1, & \mbox{if $n$ is even}, \\ [0.1cm]
5n^{2}-n+6, & \mbox{if $n$ is odd}.
\end{array}
\right.
\end{equation}
\end{theorem}
\begin{proof}~We consider the following two cases.

\textsf{Case-1:}~$n$ is even.~~In this case, we set the subgraph induced by vertex set \{$(u_{n/2},v_{1})$, $(u_{n/2},v_{2})$ ,..., $(u_{n/2},v_{10})$, $(u_{n/2+1},v_{1})$, $(u_{n/2+1},v_{2})$ ,..., $(u_{n/2+1},v_{10})$\} of $P_{n} \square P$ as $L_{0}$ then $\diam(L_{0})$ = $k$ = 3 and the maximum level in $P_{n} \square P$ is $h = n/2-1$. Note that $p$ = $10n$ and $\sum_{i=0}^{h}|L_{i}|i$ = $5n(n-2)/2$. Substituting these all in (\ref{eqn:lower}), we obtain $\rn(P_{n} \square P) \geq 5n^{2}-n+1$.

We now prove that this lower bound is tight. Note that for this purpose, it suffices to give a radio labeling $\vp$ of $P_{n} \square P$ with span equal to this lower bound and for this, it is enough to give a radio labeling satisfying conditions of Theorem \ref{thm:main}. We first order the vertices of $P_{n} \square P$ and define recursive formula of radio labeling $\vp$ on it. Let $\alpha = \bigl(\begin{smallmatrix}
  1 & 2 & 3 & 4 & 5 & 6 & 7 & 8 & 9 & 10 \\
  1 & 8 & 3 & 7 & 2 & 10 & 5 & 4 & 6 & 9
\end{smallmatrix}\bigr)$, $\beta = \bigl(\begin{smallmatrix}
  1 & 2 & 3 & 4 & 5 & 6 & 7 & 8 & 9 & 10 \\
  9 & 1 & 10 & 3 & 7 & 2 & 4 & 6 & 5 & 8
\end{smallmatrix}\bigr)$, $\sigma = \bigl(\begin{smallmatrix}
  1 & 2 & 3 & 4 & 5 & 6 & 7 & 8 & 9 & 10 \\
  2 & 9 & 1 & 8 & 3 & 7 & 6 & 5 & 4 & 10
\end{smallmatrix}\bigr)$ and $\tau = \bigl(\begin{smallmatrix}
  1 & 2 & 3 & 4 & 5 & 6 & 7 & 8 & 9 & 10 \\
  7 & 2 & 8 & 1 & 10 & 3 & 5 & 4 & 6 & 9
\end{smallmatrix}\bigr)$ be four permutations. Using these four permutations, we first rename $(u_{i},v_{j})(1 \leq i \leq n,1 \leq j \leq 10)$ as $(a_{r},b_{s})$ as follows:
\begin{equation*}
(a_{r},b_{s}) := \left\{
\begin{array}{ll}
(u_{i},v_{\alpha(j)}), & \mbox{if $1 \leq i \leq n/2$ and $(n/2-i) \equiv 0$ (mod 2)}, \\ [0.1cm]
(u_{i},v_{\beta(j)}), & \mbox{if $1 \leq i \leq n/2$ and $(n/2-i) \equiv 1$ (mod 2)}, \\ [0.1cm]
(u_{i},v_{\sigma(j)}), & \mbox{if $n/2+1 \leq i \leq n$ and $(n-i) \equiv 0$ (mod 2)}, \\ [0.1cm]
(u_{i},v_{\tau(j)}), & \mbox{if $n/2+1 \leq i \leq n$ and $(n-i) \equiv 1$ (mod 2)}.
\end{array}
\right.
\end{equation*}
We now define an ordering $x_{0},x_{1},...,x_{p-1}$ as follows: Let $x_{t} := (a_{r},b_{s})$, where
\begin{equation*}
t := \left\{
\begin{array}{ll}
(n/2-r)20+2(s-1), & \mbox{if $1 \leq r \leq n/2$}, \\ [0.1cm]
(n-r)20+2s-1, & \mbox{if $n/2+1 \leq r \leq n$}.
\end{array}
\right.
\end{equation*}
Then note that $x_{0},x_{p-1} \in L_{0}$ and for $0 \leq i \leq p-2$, $d(x_{i},x_{i+1})$ = $d(x_{i},L_{0})+d(x_{i+1},L_{0})+k$. Define $\vp$ as $\vp(x_{0}) = 0$, $\vp(x_{i+1})$ = $\vp(x_{i})+d+1-d(x_{i},L_{0})-d(x_{i+1},L_{0})-k$.

\textsf{Claim-1:} $\vp$ is a radio labeling with $\Span(\vp)$ = $5n^{2}-n+1$.

Let $x_{i}$ and $x_{j}, 0 \leq i < j \leq p-1$ be two arbitrary vertices. If $j = i+1$ then $\vp(x_{j})-\vp(x_{i})$ = $d+1-d(x_{i},L_{0})-d(x_{i+1},L_{0})-k$ = $d+1-d(x_{i},x_{i+1})$. If $j \geq i+4$ then $\vp(x_{j})-\vp(x_{i}) \geq (j-i)(d-k+1)-\sum_{t=i+1}^{j-1}d(x_{t},L_{0})-d(x_{i},L_{0})-d(x_{j},L_{0}) \geq 4(n-1)-(n-2)/2-n/2-(n-2)/2-n/2 > n+2 >n+2-d(x_{i},x_{j}) = d+1-d(x_{i},x_{j})$. If $j = i+3$ then $\vp(x_{j})-\vp(x_{i}) = (j-i)(d-k+1)-\sum_{t=i+1}^{j-1}d(x_{t},L_{0})-d(x_{i},L_{0})-d(x_{j},L_{0}) \geq 3(n-1)-n/2-(n-2)/2-(n-2)/2 = (3n-2)/2 \geq n+1 \geq n+2-d(x_{i},x_{j}) = d+1-d(x_{i},x_{j})$ as $d(x_{i},x_{j}) \geq 1$. If $j=i+2$ then $\vp(x_{j})-\vp(x_{i})$ = $(j-i)(d-k+1)-d(x_{i},L_{0})-2d(x_{i+1},L_{0})-d(x_{i+2},L_{0})$. If (1) $d(x_{i},L_{0})+2d(x_{i+1},L_{0})+d(x_{i+2},L_{0})$ = $n-1$ then $d(x_{i},x_{j})$ = 3 and hence $\vp(x_{j})-\vp(x_{i})$ = $2(n-1)-(n-1)$ = $(n-1)$ = $n+2-d(x_{i},x_{j})$ = $d+1-d(x_{i},x_{j})$. (2) $d(x_{i},L_{0})+2d(x_{i+1},L_{0})+d(x_{i+2},L_{0})$ = $n-2$ then $d(x_{i},x_{j})$ = 2 and hence $\vp(x_{j})-\vp(x_{i})$ = $2(n-1)-(n-2)$ = $n$ = $n+2-d(x_{i},x_{j})$ = $d+1-d(x_{i},x_{j})$. Hence, $\vp$ is a radio labeling. The span of $\vp$ is $\Span(\vp)$ = $\vp(x_{p-1})-\vp(x_{0})$ =
$\sum_{t=0}^{p-1}(\vp(x_{t+1})-\vp(x_{t}))$ = $(p-1)(d-k+1)-2\sum_{t=0}^{p-1}d(x_{t},L_{0})$ = $(p-1)(d-k+1)-2\sum_{i=0}^{h}|L_{i}|i$ which is equal to $5n^{2}-n+1$ in the present case.

\textsf{Case-2:}~$n$ is odd.~~In this case, we set the subgraph induced by vertex set $\{(u_{(n+1)/2},v_{1})$, $(u_{(n+1)/2},v_{2})$,..., $(u_{(n+1)/2},v_{10})\}$ of $P_{n} \square P$ as $L_{0}$ then $\diam(L_{0})$ = $k$ = 2 and the maximum level in $P_{n} \square P$ is $h = n/2-1$. Note that $p$ = $10n$ and $\sum_{i=0}^{h}|L_{i}|i$ = $5(n^{2}-1)/2$. Substituting these all in \eqref{eqn:lower}, we obtain $\rn(P_{n} \square P) \geq 5n^{2}-n+5$. Now if possible then assume that $\rn(P_{n} \square P) = 5n^{2}-n+5$ then there exist a radio labeling $\vp$ of $P_{n} \square P$ with $\Span(\vp)$ = $5n^{2}-n+5$. By Theorem \ref{thm:main}, $\vp$ induces an ordering $x_{0},x_{1},...,x_{p-1}$ of $V(P_{n} \square P)$ with $0 = \vp(x_{0}) < \vp(x_{1}) < ... < \vp(x_{p-1}) = \Span(\vp)$ which satisfies (a), (b) and (c) of Theorem \ref{thm:main}. Let $L$ = $\{(u_{1},v_{1}),(u_{1},v_{2}),...,(u_{1},v_{10})\}$, $C$ = \{$(u_{(n+1)/2},v_{1})$, $(u_{(n+1)/2},v_{2})$ ,..., $(u_{(n+1)/2},v_{10})$\} and $R$ = $\{(u_{n},v_{1})$, $(u_{n},v_{2})$,..., $(u_{n},v_{10})\}$. Since $|L|$ = $|R|$ = $|C|$ and $\vp$ satisfies conditions (a), (b) and (c) of Theorem \ref{thm:main}, there exist a vertex $x_{t} \in L$ or $R$ such that $d(x_{t-1},L_{0})+d(x_{t},L_{0}) > (n-1)/2$ and $d(x_{t},L_{0})+d(x_{t+1},L_{0}) > (n-1)/2$. Without loss of generality, assume that $d(x_{t-1},L_{0})+d(x_{t},L_{0}) \geq d(x_{t},L_{0})+d(x_{t+1},L_{0})$. Since an ordering $x_{0},x_{1},...,x_{p-1}$ of $V(P_{n} \square P)$ satisfies condition (a) of Theorem \ref{thm:main}, it is clear that $d(x_{t-1},x_{t+1})$ = $d(x_{t-1},L_{0})-d(x_{t+1},L_{0})+2$. Now consider $\vp(x_{t+1})-\vp(x_{t-1})$ = $\vp(x_{t+1})-\vp(x_{t})+\vp(x_{t})-\vp(x_{t-1})$ = $n+2-d(x_{t+1},L_{0})-d(x_{t},L_{0})-2+n+2-d(x_{t},L_{0})-d(x_{t-1},L_{0})-2$ = $2n-(d(x_{t-1},L_{0})-d(x_{t+1},L_{0})+2)-2(d(x_{t},L_{0})+d(x_{t+1},L_{0})-1) \leq 2n-d(x_{t-1},x_{t+1})-2((n+1)/2-1)$ = $n+1-d(x_{t-1},x_{t+1}) < n+2-d(x_{t-1},x_{t+1})$, a contradiction with $\vp$ is a radio labeling. Hence, $\rn(P_{n} \square P) \geq 5n^{2}-n+6$. We now prove that this lower bound is the actual value for $\rn(P_{n} \square P)$. Note that for this purpose, it is enough to give a radio labeling $\vp$ of $P_{n} \square P$ with $\Span(\vp)$ = $5n^{2}-n+6$. We order the vertices of $P_{n} \square P$ and define recursive formula on this ordering for $\vp$. We consider the following two cases.

\textsf{Subcase-2.1:} $n \equiv 1$ (mod 4).

Let $\alpha = \bigl(\begin{smallmatrix}
  1 & 2 & 3 & 4 & 5 & 6 & 7 & 8 & 9 & 10 \\
  1 & 4 & 3 & 6 & 2 & 7 & 9 & 8 & 10 & 5
\end{smallmatrix}\bigr)$,
$\beta = \bigl(\begin{smallmatrix}
  1 & 2 & 3 & 4 & 5 & 6 & 7 & 8 & 9 & 10 \\
  2 & 7 & 1 & 5 & 3 & 6 & 8 & 10 & 9 & 4
\end{smallmatrix}\bigr)$, $\sigma = \bigl(\begin{smallmatrix}
  1 & 2 & 3 & 4 & 5 & 6 & 7 & 8 & 9 & 10 \\
  2 & 3 & 1 & 7 & 4 & 5 & 6 & 9 & 10 & 8
\end{smallmatrix}\bigr)$ be three permutations. Using these three permutations, we first rename $(u_{i},v_{j})$, $(1 \leq i \leq n, 1 \leq j \leq 10)$ as $(a_{r},b_{s})$ as follows:
\begin{equation*}
(a_{r},b_{s}) := \left\{
\begin{array}{ll}
(u_{i},v_{\alpha(j)}), & \mbox{if $i$ = $(n+1)/2$}, \\ [0.2cm]
(u_{i},v_{\sigma^{2(n-i)}\beta(j)}), & \mbox{if $(n+1)/2 < i \leq n$}, \\ [0.2cm]
(u_{i},v_{\sigma^{2((n+1)/2-i)+1}\beta(j)}), & \mbox{if $1 \leq i < (n+1)/2$}.
\end{array}
\right.
\end{equation*}
We now define an ordering $x_{0},x_{1},...,x_{p-1}$ as follows: Set $x_{0} = (a_{(n+1)/2},b_{1})$, $x_{p-1} = (a_{(n+1)/2},b_{10})$ and for $1 \leq t \leq p-2$, let $x_{t} := (a_{r},b_{s})$, where
\begin{equation*}
t := \left\{
\begin{array}{ll}
(n+1-2r)+n(s-1), & \mbox{if $1 \leq r \leq (n+1)/2$ and $1 \leq s \leq 7$}, \\ [0.1cm]
2(n-r)+n(s-1)+1, & \mbox{if $(n+1)/2 < r \leq n$ and $1 \leq s \leq 7$},\\ [0.1cm]
(n+1-2r)+n(s-1)-1, & \mbox{if $1 \leq r < (n+1)/2$ and $8 \leq s \leq 10$}, \\ [0.1cm]
2(n-r)+n(s-1), & \mbox{if $(n+1)/2 < r \leq n$ and $8 \leq s \leq 10$}, \\ [0.1cm]
ns-1, & \mbox{if $r = (n+1)/2$ and $8 \leq s \leq 10$}.
\end{array}
\right.
\end{equation*}
Then note that $x_{0},x_{p-1} \in L_{0}$ and for $0 \leq i \leq p-2$, $d(x_{i},x_{i+1})$ = $d(x_{i},L_{0})+d(x_{i+1},L_{0})+k$. Define $\vp$ as follows: $\vp(x_{0}) =0$ and $\vp(x_{i+1})$ = $\vp(x_{i})+d+1-d(x_{i},L_{0})-d(x_{i+1},L_{0})-k$ for $0 \leq i \leq p-2, i \neq p-3n-1$ and $\vp(x_{p-3n})$ = $\vp(x_{p-3n-1})+d+1-d(x_{i},L_{0})-d(x_{i+1},L_{0})-k+1$.

\textsf{Claim-2:} $\vp$ is a radio labeling with $\Span(\vp)$ = $5n^{2}-n+6$.

Let $x_{i}$ and $x_{j}, 0 \leq i < j \leq p-1$ be two arbitrary vertices. If $j = i+1$ then it is clear that $\vp(x_{j})-\vp(x_{i}) \geq d+1-d(x_{i},L_{0})-d(x_{j},L_{0})-k$ = $d+1-d(x_{i},x_{j})$. If $j \geq i+3$ then if (1) $0 \leq i \leq p-3n-4$ or $p-3n \leq i \leq p-4$ then $\vp(x_{j})-\vp(x_{i}) \geq (j-i)(d-k+1)-2\sum_{t=i+1}^{j-1}d(x_{t},L_{0})-d(x_{i},L_{0})-d(x_{j},L_{0}) = 3n-(d(x_{i},L_{0})+d(x_{i+1},L_{0}))+(d(x_{i+1},L_{0})+d(x_{i+2},L_{0}))+(d(x_{i+2},L_{0})+d(x_{i+3},L_{0})) \geq 3n-(n+1)/2-(n-1)/2-(n+1)/2 = (3n-1)/2 > n+1 > n+2-d(x_{i},x_{j}) = d+1-d(x_{i},x_{j})$ as $d(x_{i},x_{j}) \geq 1$; (2) $i \in \{p-3n-3,p-3n-2,p-3n-1\}$ then $\vp(x_{j})-\vp(x_{i}) \geq (j-i)(d-k+1)-2\sum_{t=i+1}^{j-1}d(x_{t},L_{0})-d(x_{i},L_{0})-d(x_{j},L_{0})+1 = 3n-(d(x_{i},L_{0})+d(x_{i+1},L_{0}))+(d(x_{i+1},L_{0})+d(x_{i+2},L_{0}))+(d(x_{i+2},L_{0})+d(x_{i+3},L_{0}))+1 \geq 3n-(n+1)/2-(n-1)-(n+1)/2+1 = n+1 > n+2-d(x_{i},x_{j}) = d+1-d(x_{i},x_{j})$ as $d(x_{i},x_{j}) \geq 1$. If $j = i+2$ then if (1) $0 \leq i \leq p-3n-3$ or $p-3n \leq i \leq p-3$ then $\vp(x_{j})-\vp(x_{i}) = (j-i)(d-k+1)-( d(x_{i},L_{0})+d(x_{i+1},L_{0}))-(d(x_{i+1},L_{0})+d(x_{i+1},L_{0})) \geq 2n-(n+1)/2-(n-1)/2 = n \geq n+2-d(x_{i},x_{j}) = d+1-d(x_{i},x_{j})$ as $d(x_{i},x_{j}) \geq 2$; (2) $i \in \{p-3n-2,p-3n-1\}$ then it is easy to verify $\vp(x_{j})-\vp(x_{i}) \geq  n+2-d(x_{i},x_{j}) = d+1-d(x_{i},x_{j})$. Hence, $\vp$ is a radio labeling. The span of $\vp$ is $\Span(\vp) = (p-1)(d-k+1)-2\sum_{i=0}^{h}|L_{i}|i+1$ which is equal to $5n^{2}-n+6$ in the present case.

\textsf{Subcase-2.2:} $n \equiv 3$ (mod 4).

Let $\alpha = \bigl(\begin{smallmatrix}
  1 & 2 & 3 & 4 & 5 & 6 & 7 & 8 & 9 & 10 \\
  1 & 7 & 2 & 9 & 3 & 8 & 5 & 6 & 4 & 10
\end{smallmatrix}\bigr)$ and $\sigma = \bigl(\begin{smallmatrix}
  1 & 2 & 3 & 4 & 5 & 6 & 7 & 8 & 9 & 10 \\
  3 & 1 & 2 & 6 & 4 & 5 & 8 & 9 & 10 & 7
\end{smallmatrix}\bigr)$ be two permutations. Using these two permutations, we first rename $(u_{i},v_{j}), (1 \leq i \leq n, 1 \leq j \leq 10)$ as $(a_{r},b_{s})$ as follows:
\begin{equation*}
(a_{r},b_{s}) := \left\{
\begin{array}{ll}
(u_{i},v_{\alpha(j)}), & \mbox{if $i$ = $(n+1)/2$}, \\ [0.2cm]
(u_{i},v_{\sigma^{2(n-i)+1}\alpha(j)}), & \mbox{if $(n+1)/2 < i \leq n$ }, \\ [0.2cm]
(u_{i},v_{\sigma^{2((n+1)/2-i)-1}\alpha(j)}), & \mbox{if $1 \leq i < (n+1)/2$}.
\end{array}
\right.
\end{equation*}
We now define an ordering $x_{0},x_{1},...,x_{p-1}$ as follows: Set $x_{0} = (a_{(n+1)/2},b_{1})$, $x_{p-1} = (a_{(n+1)/2},b_{10})$ and for $1 \leq t \leq p-2$, let $x_{t} := (a_{r},b_{s})$, where
\begin{equation*}
t := \left\{
\begin{array}{ll}
(n+1-2r)+n(s-1), & \mbox{if $1 \leq r \leq (n+1)/2$ and $1 \leq s \leq 9$}, \\ [0.1cm]
2(n-r)+n(s-1)+1, & \mbox{if $(n+1)/2 < r \leq n$ and $1 \leq s \leq 9$},\\ [0.1cm]
(n+1-2r)+n(s-1)-1, & \mbox{if $1 \leq r < (n+1)/2$ and $s = 10$}, \\ [0.1cm]
2(n-r)+n(s-1), & \mbox{if $(n+1)/2 < r \leq n$ and $s = 10$}.
\end{array}
\right.
\end{equation*}
Then note that $x_{0},x_{p-1} \in L_{0}$ and for $0 \leq i \leq p-2$, $d(x_{i},x_{i+1})$ = $d(x_{i},L_{0})+d(x_{i+1},L_{0})+k$. Define $\vp$ as follows: $\vp(x_{0})$ = 0, $\vp(x_{i+1})$ = $\vp(x_{i})+d+1-d(x_{i},L_{0})-d(x_{i+1},L_{0})-k$ for $0 \leq i \leq p-2, i \neq p-n-1$ and $\vp(x_{p-n})$ = $\vp(x_{p-n-1})+d+1-d(x_{i},L_{0})-d(x_{i+1},L_{0})-k+1$.

\textsf{Claim-3:} $\vp$ is a radio labeling with $\Span(\vp)$ = $5n^{2}-n+6$.

Let $x_{i}$ and $x_{j}, 0 \leq i < j \leq p-1$ be two arbitrary vertices. If $j = i+1$ then it is clear that $\vp(x_{j})-\vp(x_{i}) \geq d+1-d(x_{i},L_{0})-d(x_{j},L_{0})-k$ = $d+1-d(x_{i},x_{j})$. If $j \geq i+3$ then if (1) $0 \leq i \leq p-n-4$ or $p-n \leq i \leq p-4$ then $\vp(x_{j})-\vp(x_{i}) \geq (j-i)(d-k+1)-2\sum_{t=i+1}^{j-1}d(x_{t},L_{0})-d(x_{i},L_{0})-d(x_{j},L_{0}) = 3n-(d(x_{i},L_{0})+d(x_{i+1},L_{0}))+(d(x_{i+1},L_{0})+d(x_{i+2},L_{0}))+(d(x_{i+2},L_{0})+d(x_{i+3},L_{0})) \geq 3n-(n+1)/2-(n-1)/2-(n+1)/2 = (3n-1)/2 > n+1 > n+2-d(x_{i},x_{j}) = d+1-d(x_{i},x_{j})$ as $d(x_{i},x_{j}) \geq 1$; (2) $i \in \{p-n-3,p-n-2,p-n-1\}$ then $\vp(x_{j})-\vp(x_{i}) \geq (j-i)(d-k+1)-2\sum_{t=i+1}^{j-1}d(x_{t},L_{0})-d(x_{i},L_{0})-d(x_{j},L_{0})+1 = 3n-(d(x_{i},L_{0})+d(x_{i+1},L_{0}))+(d(x_{i+1},L_{0})+d(x_{i+2},L_{0}))+(d(x_{i+2},L_{0})+d(x_{i+3},L_{0}))+1 \geq 3n-(n+1)/2-(n-1)-(n+1)/2+1 = n+1 > n+2-d(x_{i},x_{j}) = d+1-d(x_{i},x_{j})$ as $d(x_{i},x_{j}) \geq 1$. If $j = i+2$ then if (1) $0 \leq i \leq p-n-3$ or $p-n \leq i \leq p-3$ then $\vp(x_{j})-\vp(x_{i}) = (j-i)(d-k+1)-(d(x_{i},L_{0})+d(x_{i+1},L_{0}))-(d(x_{i+1},L_{0})+d(x_{i+1},L_{0})) \geq 2n-(n+1)/2-(n-1)/2 = n \geq n+2-d(x_{i},x_{j}) = d+1-d(x_{i},x_{j})$ as $d(x_{i},x_{j}) \geq 2$; (2) $i \in \{p-n-2,p-n-1\}$ then it is easy to verify $\vp(x_{j})-\vp(x_{i}) \geq  n+2-d(x_{i},x_{j}) = d+1-d(x_{i},x_{j})$. Hence, $\vp$ is a radio labeling. The span of $\vp$ is $\Span(\vp) = (p-1)(d-k+1)-2\sum_{i=0}^{h}|L_{i}|i+1$ which is equal to $5n^{2}-n+6$ in the present case.
\end{proof}

\begin{example}{\rm In Table 1, an ordering and the corresponding optimal radio labeling of vertices of $P_{6} \square P$ is shown.}
\end{example}
\begin{table}[h!]\label{table1}
\centering
\caption{An ordering and optimal radio labeling for vertices of $P_{6} \square P$.}
\begin{tabular}{|c|ll|lr|lr|lr|lr|lr|}
\hline
$(u_{i},v_{j}) \frac{i \rightarrow}{j \downarrow}$ &  \multicolumn{2}{|c|}{1} & \multicolumn{2}{|c|}{2} &  \multicolumn{2}{|c|}{3} & \multicolumn{2}{|c|}{4} &  \multicolumn{2}{|c|}{5} & \multicolumn{2}{|c|}{6} \\ \hline\hline
1 & $x_{40}$ & $118$ & $x_{36}$ & $107$ & \underline{\textbf{$x_{0}$}} & \underline{\textbf{0}} & $x_{43}$ & $127$ & $x_{33}$ & $98$ & $x_{3}$ & $9$ \\
2 & $x_{54}$ & 160 & $x_{20}$ & 59 & $x_{14}$ & 42 & $x_{57}$ & 169 & $x_{23}$ & 68 & $x_{17}$ & 51  \\
3 & $x_{44}$ & 130 & $x_{38}$ & 113 & $x_{4}$ & 12 & $x_{41}$ & 121 & $x_{35}$ & 104 & $x_{1}$ & 3  \\
4 & $x_{52}$ & 154 & $x_{24}$ & 71 & $x_{12}$ & 36 & $x_{55}$ & 163 & $x_{21}$ & 62 & $x_{15}$ & 45  \\
5 & $x_{42}$ & 124 & $x_{32}$ & 95 & $x_{2}$ & 6 & $x_{45}$ & 133 & $x_{39}$ & 116 & $x_{5}$ & 15  \\
6 & $x_{58}$ & 172 & $x_{22}$ & 65 & $x_{18}$ & 54 & $x_{53}$ & 157 & $x_{25}$ & 74 & $x_{13}$ & 39  \\
7 & $x_{48}$ & 142 & $x_{26}$ & 77 & $x_{8}$ & 24 & $x_{51}$ & 151 & $x_{29}$ & 86 & $x_{11}$ & 33  \\
8 & $x_{46}$ & 136 & $x_{30}$ & 89 & $x_{6}$ & 18 & $x_{49}$ & 145 & $x_{27}$ & 80 & $x_{9}$ & 27  \\
9 & $x_{50}$ & 148 & $x_{28}$ & 83 & $x_{10}$ & 30 & $x_{47}$ & 139 & $x_{31}$ & 92 & $x_{7}$ & 21  \\
10 & $x_{56}$ & 166 & $x_{34}$ & 101 & $x_{16}$ & 48 & \underline{\textbf{$x_{59}$}} & \underline{\textbf{175}} & $x_{37}$ & 110 & $x_{19}$ & 57  \\
\hline
\end{tabular}
\end{table}

\begin{example}{\rm In Table 2, an ordering and the corresponding optimal radio labeling of vertices of $P_{5} \square P$ is shown.}
\end{example}
\begin{table}[h!]\label{table1}
\centering
\caption{An ordering and optimal radio labeling for vertices of $P_{5} \square P$.}
\begin{tabular}{|c|lr|lr|lr|lr|lr|}
\hline
$(u_{i},v_{j}) \frac{i \rightarrow}{j \downarrow}$ &  \multicolumn{2}{|c|}{1} & \multicolumn{2}{|c|}{2} &  \multicolumn{2}{|c|}{3} & \multicolumn{2}{|c|}{4} &  \multicolumn{2}{|c|}{5} \\ \hline\hline
1 & $x_{9}$ & 23 & $x_{12}$ & 31 & \underline{\textbf{$x_{0}$}} & \underline{\textbf{0}} & $x_{3}$ & 8 & $x_{6}$ & 16 \\
2 & $x_{19}$ & 49 & $x_{27}$ & 70 & $x_{15}$ & 39 & $x_{23}$ & 60 & $x_{31}$ & 81 \\
3 & $x_{4}$ & 10 & $x_{7}$ & 18 & $x_{10}$ & 26 & $x_{13}$ & 34 & $x_{1}$ & 3 \\
4 & $x_{29}$ & 75 & $x_{17}$ & 44 & $x_{25}$ & 65 & $x_{33}$ & 86 & $x_{21}$ & 55 \\
5 & $x_{14}$ & 36 & $x_{2}$ & 5 & $x_{5}$ & 13 & $x_{8}$ & 21 & $x_{11}$ & 29 \\
6 & $x_{34}$ & 88 & $x_{22}$ & 57 & $x_{30}$ & 78 & $x_{18}$ & 47 & $x_{26}$ & 68 \\
7 & $x_{38}$ & 97 & $x_{41}$ & 105 & $x_{44}$ & 113 & $x_{47}$ & 121 & \underline{$x_{35}$} & \underline{90} \\
8 & $x_{48}$ & 123 & $x_{36}$ & 92 & $x_{39}$ & 100 & $x_{42}$ & 108 & $x_{45}$ & 116 \\
9 & $x_{43}$ & 110 & $x_{46}$ & 118 & \underline{\textbf{$x_{49}$}} & \underline{\textbf{126}} & $x_{37}$ & 95 & $x_{40}$ & 103 \\
10 & $x_{24}$ & 62 & $x_{32}$ & 83 & $x_{20}$ & 52 & $x_{28}$ & 73 & $x_{16}$ & 42 \\
\hline
\end{tabular}
\end{table}

\begin{example}{\rm In Table 3, an ordering and the corresponding optimal radio labeling of vertices of $P_{7} \square P$ is shown.}
\end{example}
\begin{table}[h!]\label{table3}
\centering
\caption{An ordering and optimal radio labeling for vertices of $P_{7} \square P$.}
\begin{tabular}{|c|lr|lr|lr|lr|lr|lr|lr|}
\hline
$(u_{i},v_{j}) \frac{i \rightarrow}{j \downarrow}$ &  \multicolumn{2}{|c|}{1} & \multicolumn{2}{|c|}{2} &  \multicolumn{2}{|c|}{3} & \multicolumn{2}{|c|}{4} &  \multicolumn{2}{|c|}{5} & \multicolumn{2}{|c|}{6} & \multicolumn{2}{|c|}{7} \\ \hline\hline
1 & $x_{6}$ & 21 & $x_{18}$ & 64 & $x_{9}$ & 32 & \underline{\textbf{$x_{0}$}} & \underline{\textbf{\textbf{0}}} & $x_{12}$ & 43 & $x_{3}$ & 11 & $x_{15}$ & 54 \\
2 & $x_{62}$ & 221 & $x_{46}$ & 164 & $x_{58}$ & 207 & $x_{42}$ & 150 & $x_{54}$ & 193 & $x_{65}$ & 230 & $x_{50}$ & 179 \\
3 & $x_{13}$ & 46 & $x_{4}$ & 14 & $x_{16}$ & 57 & $x_{7}$ & 25 & $x_{19}$ & 68 & $x_{10}$ & 36 & $x_{1}$ & 4 \\
4 & $x_{48}$ & 171 & $x_{60}$ & 214 & $x_{44}$ & 157 & $x_{56}$ & 200 & $x_{67}$ & 237 & $x_{52}$ & 186 & \underline{$x_{63}$} & \underline{223} \\
5 & $x_{20}$ & 71 & $x_{11}$ & 39 & $x_{2}$ & 7 & $x_{14}$ & 50 & $x_{5}$ & 18 & $x_{17}$ & 61 & $x_{8}$ & 29 \\
6 & $x_{68}$ & 240 & $x_{53}$ & 189 & $x_{64}$ & 226 & $x_{49}$ & 175 & $x_{61}$ & 218 & $x_{45}$ & 161 & $x_{57}$ & 204 \\
7 & $x_{34}$ & 121 & $x_{25}$ & 89 & $x_{37}$ & 132 & $x_{28}$ & 100 & $x_{40}$ & 143 & $x_{31}$ & 111 & $x_{22}$ & 79 \\
8 & $x_{41}$ & 146 & $x_{32}$ & 114 & $x_{23}$ & 82 & $x_{35}$ & 125 & $x_{26}$ & 93 & $x_{38}$ & 136 & $x_{29}$ & 104 \\
9 & $x_{27}$ & 96 & $x_{39}$ & 139 & $x_{30}$ & 107 & $x_{21}$ & 75 & $x_{33}$ & 118 & $x_{24}$ & 86 & $x_{36}$ & 129 \\
10 & $x_{55}$ & 196 & $x_{66}$ & 233 & $x_{51}$ & 182 & \underline{\textbf{$x_{69}$}} & \underline{\textbf{244}} & $x_{47}$ & 168 & $x_{59}$ & 211 & $x_{43}$ & 154 \\
\hline
\end{tabular}
\end{table}

\subsection{Radio number for $P_{n} \square K_{m}$}
In this section, using Theorem \ref{thm:lower} and \ref{thm:main}, we give a short proof for the radio number of $P_{n} \square K_{m}$ given by Kim \emph{et al.} in \cite{Kim}.

We assume that $m \geq 3$ and $n \geq 4$. Note that $|P_{n} \square K_{m}|$ = $|P_{n}| \times |K_{m}|$ = $nm$ and $\diam(P_{n} \square K_{m})$ = $n$. We denote the vertex set of $P_{n}$ by $V(P_{n})$ = \{$u_{1},u_{2},...,u_{n}$\} with ($u_{i}, u_{i+1}$) $\in E(P_{n})$, $1 \leq i \leq n-1$ and the vertex set of $K_{m}$ by $V(K_{m})$ = \{$v_{1},v_{2},...,v_{m}$\} with ($v_{i},v_{j}$) $\in E(K_{m})$, $1 \leq i, j \leq m, i \neq j$ then the vertex set of $P_{n} \square K_{m}$ is $V(P_{n} \square K_{m})$ = \{$(u_{i},v_{j}) : 1 \leq i \leq n, 1 \leq j \leq m$\}.

\begin{theorem} Let $m \geq 3$ and $n \geq 4$ be integers. Then
\begin{equation}\label{eqn:pnkm}
\rn(P_{n} \square K_{m}) := \left\{
\begin{array}{ll}
\frac{mn^{2}-2n+2}{2}, & \mbox{if $n$ is even}, \\ [0.1cm]
\frac{mn^{2}-2n+m+2}{2}, & \mbox{if $n$ is odd}.
\end{array}
\right.
\end{equation}
\end{theorem}
\begin{proof} We consider the following two cases.

\textsf{Case-1}: $n$ is even.~~In this case, we set the subgraph induced by vertex set \{$(u_{n/2},v_{1})$, $(u_{n/2},v_{2})$ ,..., $(u_{n/2},v_{m})$, $(u_{n/2+1},v_{1})$, $(u_{n/2+1},v_{2})$ ,..., $(u_{n/2+1},v_{m})$\} of $P_{n} \square K_{m}$ as $L_{0}$ then $\diam(L_{0})$ = $k$ = 2 and the maximum level in $P_{n} \square K_{m}$ is $h$ = $n/2-1$. Note that $p = mn$ and $\sum_{i=0}^{h}|L_{i}|i$ = $mn(n-2)/2$. Substituting these all in (\ref{eqn:lower}), we obtain $\rn(P_{n} \square K_{m}) \geq (mn^{2}-2n+2)/2$. In fact, this lower bound is the actual value for $\rn(P_{n} \square K_{m})$ and for that, it is enough to give a radio labeling with span equal to this lower bound. Note that the radio labeling given by Kim \emph{et al.} in \cite{Kim} serve this purpose (the readers are required to understand and adjust with notation matter) which complete the proof.

\textsf{Case-2}: $n$ is odd.~~In this case, we set the subgraph induced by vertex set $\{(u_{(n+1)/2},v_{1})$, $(u_{(n+1)/2},v_{2})$ ,..., $(u_{(n+1)/2},v_{m})\}$ in $P_{n} \square K_{m}$ as $L_{0}$ then $\diam(L_{0})$ = $k$ = 1 and the maximum level in $P_{n} \square K_{m}$ is $h$ = $n/2-1$. Note that $p = nm$ and $\sum_{i=0}^{h}|L_{i}|i$ = $m(n^{2}-1)/4$. Substituting these all in (\ref{eqn:lower}), we obtain $\rn(P_{n} \square K_{m}) \geq (mn^{2}-2n+m)/2$. Now if possible then assume that $\rn(P_{n} \square K_{m}) = (mn^{2}-2n+m)/2$ then there exist a radio labeling $\vp$ of $P_{n} \square K_{m}$ with $\Span(\vp)$ = $(mn^{2}-2n+m)/2$. By Theorem \ref{thm:main}, $\vp$ induces an ordering $x_{0},x_{1},...,x_{p-1}$ of $V(P_{n} \square P)$ with $0 = \vp(x_{0}) < \vp(x_{1}) < ... < \vp(x_{p-1}) = \Span(\vp)$ which satisfies (a), (b) and (c) of Theorem \ref{thm:main}. Let $L$ = $\{(u_{1},v_{1}),(u_{1},v_{2}),...,(u_{1},v_{m})\}$, $C$ = $\{(u_{(n+1)/2},v_{1})$, $(u_{(n+1)/2}, v_{2}), ... ,(u_{(n+1)/2},v_{m})\}$ and $R$ = $\{(u_{n},v_{1}),(u_{n},v_{2}),...,(u_{n},v_{m})\}$. Since $|L|$ = $|R|$ = $|C|$ and $\vp$ satisfies conditions (a), (b) and (c) of Theorem \ref{thm:main}, there exist a vertex $x_{t} \in L$ or $R$ such that $d(x_{t-1},L_{0})+d(x_{t},L_{0}) > (n-1)/2$ and $d(x_{t},L_{0})+d(x_{t+1},L_{0}) > (n-1)/2$. Without loss of generality assume that $d(x_{t-1},L_{0})+d(x_{t},L_{0}) \geq d(x_{t},L_{0})+d(x_{t-1},L_{0})$. Since an ordering $x_{0},x_{1},...,x_{p-1}$ of $V(P_{n} \square K_{m})$ satisfies condition (a) of Theorem \ref{thm:main}, it is clear that $d(x_{t-1},x_{t+1})$ = $d(x_{t-1},L_{0})-d(x_{t+1},L_{0})+1$. Now consider $\vp(x_{t+1})-\vp(x_{t-1})$ = $\vp(x_{t+1})-\vp(x_{t})+\vp(x_{t})-\vp(x_{t-1})$ = $n+1-d(x_{t+1},L_{0})-d(x_{t},L_{0})-1+n+1-d(x_{t},L_{0})-d(x_{t-1},L_{0})-1$ = $2n-(d(x_{t-1},L_{0})-d(x_{t+1},L_{0})+1)-2(d(x_{t},L_{0})+d(x_{t+1},L_{0})-1/2)$ $\leq$ $2n-d(x_{t-1},x_{t+1})-2(n/2+1-1/2)$ = $n-1-d(x_{t-1},x_{t+1}) < n+1-d(x_{t-1},x_{t+1})$, a contradiction with $\vp$ is a radio labeling. Hence, $\rn(P_{n} \square K_{m}) \geq (mn^{2}-2n+m+2)/2$. In fact, this lower bound is the actual value for $\rn(P_{n} \square K_{m})$ and for that, it is enough to give a radio labeling with span equal to this lower bound. Again note that the radio labeling given by Kim \emph{et al.} in \cite{Kim} serve this purpose (the readers are required to understand and adjust with notation matter) which complete the proof.
\end{proof}

\section{Concluding remarks}
In \cite{Das}, Das \emph{et al.} gave a technique to find a lower bound for the radio $k$-coloring of graphs which also cover the case of radio labeling when $k$ = $\diam(G)$. In \cite{Das}, authors fixed a vertex as $L_{0}$ when $\diam(G)$ is even and a maximal clique $C$ of $G$ as $L_{0}$ when $\diam(G)$ is odd. We remark that our approach is more useful to find a better lower bound for the radio number of graphs than one given by Das \emph{et al.} in \cite{Das} and this can be realize for the graph $P_{n} \square P$. Note that in case of $P_{n} \square P$, if we fix a vertex or a maximal clique then there is a large gap between a lower bound for the radio number of $P_{n} \square P$ and the actual value of radio number of $P_{n} \square P$. Moreover, a necessary and sufficient condition to achieve the lower is useful to determine the exact radio number of graphs. It is also possible to determine the existing radio number for complete graph $K_{n}$, wheel graph $W_{n}$, $n$-gear graph $G_{n}$, paths $P_{n}$ using Theorem \ref{thm:lower} and \ref{thm:main}.

Finally, we suggest the following further work in direction of present research work.
\begin{enumerate}
  \item Find graphs $G$ such that $\rn(P_{n} \square G)$ can be determine using Theorem \ref{thm:lower} and \ref{thm:main} (we suggest star graph, wheel graph etc. as $G$).
  \item Find graphs $G_{1}$ and $G_{2}$ such that $\rn(G_{1} \square G_{2})$ can be determine using Theorem \ref{thm:lower} and \ref{thm:main}.
  \item More generally, find graphs $G$ other than trees whose radio number can be determine using Theorem \ref{thm:lower} and \ref{thm:main}.
\end{enumerate}

\section*{Acknowledgements}
I want to express my deep gratitude to anonymous referees for kind comments and constructive suggestions.

\end{document}